\documentclass[11pt]{amsart}
\usepackage[latin1]{inputenc}
\usepackage{color}
\usepackage{bbm}
\usepackage{amsmath,amssymb}
\usepackage{tikz}
\usetikzlibrary{patterns}
\usetikzlibrary{shapes,decorations}

\newtheorem{thm}{Theorem}[section]
\newtheorem{cor}[thm]{Corollary}
\newtheorem{lem}[thm]{Lemma}
\newtheorem{prop}[thm]{Proposition}
\newtheorem{problem}{Problem}
\theoremstyle{definition}

\numberwithin{equation}{section}

\makeatletter
\newcommand{\bN}{\mathbb{N}}
\newcommand{\bR}{\mathbb{R}}

\newcommand{\supp}{\operatorname{supp}}

\newcommand{\diam}{\operatorname{diam}}
\newcommand{\conv}{\operatorname{convexhull}}

\newcommand{\dif}{\,\mathrm{d}}

\DeclareMathOperator{\card}{card}

\newcommand{\be}{\begin{equation}}
\newcommand{\ee}{\end{equation}}
\newcommand{\ba}{\begin{align}}
\newcommand{\ea}{\end{align}}
\newcommand{\baa}{\begin{align*}}
\newcommand{\eaa}{\end{align*}}

\newcommand{\logplus}{\ensuremath{\log^+}}

\begin{document}
\title[On a.e. convergence of tensor product spline projections]{On almost everywhere convergence of tensor product spline projections}
\author[M. Passenbrunner]{Markus Passenbrunner}
\address{Institute of Analysis, Johannes Kepler University Linz, Austria, 4040 Linz, Altenberger Strasse 69}
\email{markus.passenbrunner@jku.at}

\author[J. Prochno]{Joscha Prochno}
\address{School of Mathematics \& Physical Sciences, University of Hull, Cottingham Road, Hull, HU6 7RX, United Kingdom}
\email{j.prochno@hull.ac.uk}

\date{\today}
\begin{abstract}
Let $d\in\mathbb N$ and $f$ be a function in the Orlicz class $L(\logplus
L)^{d-1}$ defined on the unit cube $[0,1]^d$ in $\mathbb{R}^d$. 
Given knot sequences
$\Delta_1,\ldots,$ $\Delta_d$ on $[0,1]$, we first prove that the orthogonal
projection $P_{(\Delta_1,\dots,\Delta_d)}(f)$ onto the space of tensor product
splines with arbitrary orders $(k_1,\dots, k_d)$ and knots $\Delta_1,\ldots,
\Delta_d$ converges to $f$ almost everywhere as the mesh diameters
$|\Delta_1|,\ldots, |\Delta_{d}|$ tend to zero. This extends the 1-dimensional
result in \cite{PassenbrunnerShadrin2014} to arbitrary dimensions.

In a second step, we show that this result is optimal, i.e., given any ``bigger'' Orlicz class $X=\sigma(L)L(\logplus L)^{d-1}$ with an arbitrary function $\sigma$ tending to zero at infinity, there exists a function $\varphi\in X$ and partitions of the unit cube such that the orthogonal projections of $\varphi$ do not converge almost everywhere.
\end{abstract}

\maketitle
\renewcommand{\theenumi}{\roman{enumi}}

\section{Introduction and Main results}\label{sec:intro}

The notion of splines is originally motivated by concepts used in shipbuilding design and was first introduced by Schoenberg in his 1946 paper \cite{Schoenberg1946a} to approach problems of approximation. The particular interest in tensor product splines, besides a purely mathematical one, is due to their various applications in high-dimensional problems. For instance, in statistics they are used in non-parametric and semi-parametric multiple regression where high-dimensional vectors of covariates are considered for each observation (see, e.g., \cite{Wahba}) and in the approximation of finite window roughness penalty smoothers \cite{GreenSilverman}. In data mining they appear in predictive modeling with multivariate regression splines in form of popular MARS or MARS-like algorithms \cite{zaki2003}.
Further applications appear in problems related to high-dimensional numerical
integration.
With this paper we contribute to a better understanding of the theoretical aspects of tensor product splines.

One of the major mathematical achievements in the last years is Shadrin's proof of de Boor's conjecture \cite{Shadrin2001}, where he showed that the $\max$-norm of
the orthogonal projection $P_{\Delta}$ onto spline spaces of arbitrary order $k$ with knots $\Delta$ is bounded independently of the
knot-sequence $\Delta$. In particular, this result implies the
$L_p$-convergence ($1\leq p < \infty$) of orthogonal spline projections, i.e., for all $f\in L_p[a,b]$,
\[
  P_{\Delta}(f)\stackrel{L_p}{\longrightarrow} f,
\]
provided the mesh diameter $|\Delta|$ tends to zero.
A similar result holds for the $L_{\infty}$-norm if one replaces
the space $L_\infty$ with the space of continuous functions.
Recently, in \cite{PassenbrunnerShadrin2014}, Shadrin and the first named author extended this result. They were
able to prove that the $\max$-norm boundedness of $P_{\Delta}$ implies the almost everywhere (a.e.) convergence
of orthogonal projections $P_{\Delta}(f)$ with arbitrary knot-sequences $\Delta$ and $f\in L_1[a,b]$,
provided that the mesh diameter $|\Delta|$ tends to zero. Their proof is based on a classical approach where a.e.
convergence is proved on a dense subset of $L_1$ and where it is shown that the maximal projection operator
is of weak $(1,1)$-type. The main tool in the proof of this theorem is a sharp decay inequality for inverses
of B-spline Gram matrices.

This leaves open the natural question of a corresponding result in higher dimensions.
In a first step in this work, we extend the 1-dimensional result obtained in
\cite{PassenbrunnerShadrin2014} to
arbitrary dimensions $d\in\mathbb N$, where the function $f$, defined on the
unit cube in $\bR^d$, belongs to the Orlicz class $L(\log^+L)^{d-1}$ (details
are given below). In a second step, and this is the main result of this paper,
we prove that this is in fact optimal. 

Let us present our results in more detail. We write $P_{\bf \Delta}$ for the orthogonal projection operator from $L_2[0,1]^d$ onto the linear span of the sequence of tensor product B-splines and denote by $|{\bf\Delta}|$ the maximal directional mesh width.
The first result of this article is the a.e. convergence of $P_{\bf \Delta}f$ to $f$ for the Orlicz class $L(\logplus L)^{d-1}$:
\begin{thm}\label{thm:ae}
Let $f\in L(\logplus L)^{d-1}$.
Then, as $|{\bf\Delta}|\to 0$,
\begin{equation*}
P_{\bf\Delta}f\to f\quad\text{a.e.}.
\end{equation*}
\end{thm}

The second and main result of this work shows that this result is optimal:
\begin{thm}\label{thm:saks-extension}
For any positive function $\sigma$ on $[0,\infty)$ with $\liminf_{t\to\infty}\sigma(t)=0$, there exists a non-negative function $\varphi$ on $[0,1]^d$ such that
\begin{enumerate}
\item the function $\sigma(\varphi)\cdot\varphi\cdot(\logplus \varphi)^{d-1}$ is integrable,
\item 
	there exists a subset $B\subset [0,1]^d$ of positive Lebesgue measure and a sequence of partitions
	$({\bf\Delta}_n)$ of $[0,1]^d$ with $|{\bf\Delta}_n|\to 0$ such that, for all $x\in B$,
	\begin{equation*}
		\limsup_{n\to\infty} |P_{{\bf\Delta}_n}\varphi(x)| = \infty.
	\end{equation*}
\end{enumerate}
\end{thm}

{The paper is organized as follows. In Section \ref{sec:prelim and main results}
we present the notation and notions used throughout this work and present some
preliminary results. 
In Section \ref{sec:proofs}, 
we then give the proof of Theorem
\ref{thm:ae}. The proof of Theorem \ref{thm:saks-extension}, showing the
optimality of Theorem \ref{thm:ae}, is presented in Section
\ref{sec:negative}. 
{We conclude the paper in Section \ref{sec:final remarks} with some final
remarks and an open problem that we consider to be of further interest.

\section{Notation and Preliminaries}\label{sec:prelim and main results}

In this section we introduce the notation used throughout the text and present
some background material such as a multi-dimensional version of Remez'
inequality, which we will use later, and recall the definition of tensor product B-splines.

\subsection{General notation}

We write $\card[A]$ to denote the cardinality of a set $A$. The symbol $|\cdot|$ will be used for the modulus, the mesh width and the Lebesgue measure and the meaning as well as the dimension of the Lebesgue measure will be always clear from the context.
Given a compact metric space $M$, we denote by $C(M)$ the space of continuous functions on $M$. As usual, for $1\leq p \leq \infty$ and a measure space $(E,\Sigma,\mu)$, we denote by $L_p(E)$ the space of (equivalence classes of) measurable functions $f:E\to \mathbb R$ for which
\[
\|f\|_{L_p(E)} := \bigg(\int_E |f|^p \dif \mu\bigg)^{1/p} <\infty
\]
for $1\leq p<\infty$ and
\[
\|f\|_{L_\infty(E)} := \inf\{\rho\geq 0\,:\, \mu(|f|>\rho) = 0 \}<\infty
\]
when $p=\infty$. We will also write $\|f\|_p$ instead of
	$\|f\|_{L_p(E)}$ when the choice of $E$ is clear from the context.
	More generally, given a convex function $M:[0,\infty)\to[0,\infty)$ with
		$M(0)=0$, the set of all (equivalence classes of) measurable
		functions $f:E\to\mathbb R$ such that, for some (and thus
		for all) $\lambda>0$,
\[
\int_{E}M\left({|f|\over \lambda}\right)\,\dif \mu < \infty,
\]
is called {Orlicz space} associated with $M$ and is denoted by $L_M(E)$. This space becomes a Banach space when it is supplied with the Luxemburg norm
\[
\|f\|_{M} = \inf\bigg\{ \lambda>0 \,:\, \int_{E}M\left({|f|\over \lambda}\right)\,\dif \mu \leq 1 \bigg\}\,.
\]
In this work, we consider functions $f$ defined on the unit cube $[0,1]^d$,
which belong to the Orlicz space $L(\log^+ L)^{j}$, i.e., $|f|(\log^+ |f|)^j$ is
integrable over $[0,1]^d$ with respect to Lebesgue measure, where
$\log^+(\cdot):=\max\{0,\log(\cdot)\}$. More information and a detailed
exposition of the theory of Orlicz spaces can be found, for instance, in
\cite{KrasnoselskiiRutickii,Kosmol,RaoRen1991,RaoRen2002}.

\subsection{Remez' inequality for polynomials}
We will need the following multi-dimensional version of Remez' theorem
(see \cite{Ga2001,BrudnyiGanzburg1973}). If
	$p(x)=\sum_{\alpha\in I} a_\alpha x^\alpha$ is a $d$-variate polynomial
	where $I$ is a finite set containing $d$-dimensional multiindices, the
	degree of $p$ is defined as $\max\{\sum_{i=1}^d \alpha_i :
\alpha\in I\}$.
Recall that a convex body in $\mathbb R^d$ is a compact, convex set with non-empty interior.
\begin{thm}[Remez, Brudnyi, Ganzburg]
	Let $d\in\mathbb N$, $V\subset \mathbb R^d$ a convex body and $E\subset V$ a measurable
	subset. Then, for all polynomials $p$ of degree $k$ on $V$,
	\begin{equation*}
		\| p \|_{L_\infty(V)} \leq \bigg( 4d \frac{|V|}{|E|}\bigg)^k \| p
		\|_{L_\infty(E)}.
	\end{equation*}
\end{thm}
We have the following corollary:
\begin{cor} \label{cor:remez}
Let $p$ be a polynomial of degree $k$ on a convex body $V\subset \mathbb R^d$. Then
\begin{equation*}
	\big|\big\{ x \in V : |p(x)| \geq (8d)^{-k} \|p\|_{L_\infty(V)} \big\}\big| \geq |V|/2.
\end{equation*}
\end{cor}
\begin{proof}
	This follows from an application of the above theorem to the set $E = \{x\in
		V : |p(x)| \leq (8d)^{-k} \|p\|_{L_\infty(V)}\}$.
\end{proof}

\subsection{Tensor product B-splines}\label{subsec:Bsplines}

We will now provide some background information on tensor product splines. For more information we refer the reader to \cite[Section 12.2]{Schumaker2007}. Let $d\in \bN$ and for $\mu\in\{1,\dots,d\}$, let $k_\mu$ be the order of polynomials in
the direction of the $\mu$-th standard unit vector in $\mathbb{R}^d$,
where the order of a univariate polynomial refers to the degree plus $1$. For each such $\mu$, we define a partition of the interval $[0,1]$ by
\[
\Delta_\mu = (t_i^{(\mu)})_{i=1}^{n_\mu+k_\mu},\quad n_\mu\in\mathbb N,
\]
where, for each $i<n_\mu+k_\mu$ and $j\leq n_\mu$,
\[
t_i^{(\mu)}\leq t_{i+1}^{(\mu)}\quad\text{and}\quad  t_j^{(\mu)}<t_{j+k_{\mu}}^{(\mu)},
\]
as well as 
\[
t_1^{(\mu)}=\dots=t_{k_\mu}^{(\mu)}=0\quad\text{and}\quad 1=t_{n_\mu+1}^{(\mu)}=\dots=t_{n_\mu+k_\mu}^{(\mu)}.
\]

A boldface letter always denotes a vector of $d$ entries and its coordinates are denoted by the same letter, for instance ${\bf n}=(n_1,\dots,n_d)$, ${\bf k}=(k_1,\dots, k_d)$, or ${\bf \Delta}=(\Delta_1,\dots,\Delta_d)$.
We let $(N_i^{(\mu)})_{i=1}^{n_\mu}$ be the sequence of B-splines of order $k_\mu$ on the partition $\Delta_\mu$ with the properties
\[
\supp N_i^{(\mu)}=\Big[t_i^{(\mu)},t_{i+k_\mu}^{(\mu)}\Big]\quad\text{and}\quad N_i^{(\mu)}\geq 0\quad\text{and}\quad \sum_{i=1}^{n_\mu}N_i^{(\mu)}\equiv 1.
\]
The space that is spanned by those B-spline functions consists of piecewise polynomials $p$ of order $k_\mu$ with grid points
$\Delta_\mu$, which satisfy the following smoothness conditions at those grid points: if the point $t$ occurs $m$
times in $\Delta_\mu$, the function $p$ is $k_\mu-1-m$ times continuously differentiable at $t$. In particular, 
if $m=k_\mu$, then there is no smoothness condition at the point $t$.

The tensor product B-splines are defined as 
\[
N_{\bf i}(x_1,\dots,x_d):=  N_{i_1}^{(1)}(x_1)\cdots N_{i_d}^{(d)}(x_d),\qquad {\bf 1}\leq {\bf i}\leq {\bf n},
\]
where $\bf 1$ is the $d$-dimensional vector consisting of $d$ entries equal to one and where we say that ${\bf i}\leq {\bf n}$, provided that $i_\mu\leq n_\mu$ for all $\mu\in\{1,\dots,d\}$.
Furthermore, $P_{\bf \Delta}$ is defined to be the orthogonal projection
operator from $L_2[0,1]^d$ onto the linear span of the functions
$(N_{\bf i})_{{\bf 1}\leq {\bf i}\leq {\bf n}}$ with respect to the standard
inner product $\langle\cdot,\cdot\rangle$. This operator can be naturally extended to $L_1$-functions since B-splines are contained in $L_\infty$ (cf. Lemma \ref{lem:dirichlet} below). For $\mu\in\{1,\dots,d\}$, we define
the mesh width in the direction of $\mu$ by $|\Delta_\mu| :=\max_i \big|t_{i+1}^{(\mu)}-t_i^{(\mu)}\big|$ and the mesh width by
\[
|{\bf\Delta}|:=\max_{1\leq\mu\leq d} |\Delta_\mu|.
\]

\section{Almost everywhere convergence}\label{sec:proofs}

In this section, we prove Theorem \ref{thm:ae} on a.e. convergence. Its proof follows along the lines
of the 1-dimensional case proved in \cite{PassenbrunnerShadrin2014} and is based on the standard approach of verifying the following two conditions that imply a.e. convergence of $P_{\bf\Delta}f$ for $f\in L(\logplus L)^{d-1}$ (see \cite[pp. 3-4]{Garsia1970}):
\renewcommand{\theenumi}{\alph{enumi}}
\begin{enumerate}
	\item \label{enu:1} there is a dense subset $\mathcal F$ of $L(\logplus L)^{d-1}$ on which we have a.e. convergence,
	\item \label{enu:2} the maximal operator $P^*f:=\sup_{\bf\Delta} |P_{\bf\Delta}f|$ satisfies some weak type inequality.
\end{enumerate}
\renewcommand{\theenumi}{\roman{enumi}}
Let us now discuss the latter two conditions \eqref{enu:1} and \eqref{enu:2} in more detail. Concerning \eqref{enu:1}, we first note that for $d=1$, Shadrin proved in  \cite{Shadrin2001} that the 1-dimensional projection operator $P_{\Delta}$ is uniformly bounded on $L_\infty$ for any spline order $k$, i.e.,
\[
\|P_{\Delta}\|_\infty \leq c_{k},
\]
where the constant $c_{k}\in(0,\infty)$ depends only on $k$ and not on the partition $\Delta$. 
A direct consequence of this result and of the tensor structure of the underlying operator $P_{\bf\Delta}$ is that this assertion also holds in higher dimensions:
\begin{cor}\label{cor:uniformboundedness}
For any $d\in\bN$ there exists a constant $c_{d,\bf k}\in(0,\infty)$ that only depends on $d$ and ${\bf k}$ such that 
\[
\|P_{\bf\Delta}\|_\infty\leq c_{d,\bf k}.
\]
In particular, $c_{d,\bf k}$ is independent of the partitions $\bf\Delta$.
\end{cor}
This can be used to prove uniform convergence of $P_{\bf\Delta}g$ to $g$ for
continuous functions $g$, provided $|{\bf\Delta}|$ tends to zero:

\begin{prop}\label{prop:uniformconvergence}
Let $g\in C\big([0,1]^d\big)$. Then, as $|{\bf\Delta}|\to 0$,
\[
\|P_{\bf\Delta}g- g\|_\infty\to 0.
\]
\end{prop}
 
Therefore, we may choose $\mathcal{F}$ to be the space of continuous functions on $[0,1]^d$, which is dense in $L(\logplus L)^{d-1}$ (see, e.g., \cite[Chapter 7]{Kosmol}).

We now turn to the discussion of condition \eqref{enu:2} and define the strong maximal function ${\rm M_S}f$ of $f\in L_1[0,1]^d$ by
\[
{\rm M_S}f(x):=\sup_{I\ni x}\,\frac{1}{|I|}\int_I |f(y)|\dif y,\qquad x\in [0,1]^d,
\]
where the supremum is taken over all $d$-dimensional rectangles $I\subset [0,1]^d$, which are parallel to the coordinate axes and contain the point $x$.
The strong maximal function satisfies the weak type inequality 
\begin{equation}\label{eq:weaktype}
|\{x: {\rm M_S} f(x)>\lambda\}|\leq c_M \int_{[0,1]^d}\frac{|f(x)|}{\lambda}\bigg(1+\logplus\frac{|f(x)|}{\lambda}\bigg)^{d-1}\dif x,
\end{equation}
where $|A|$ denotes the $d$-dimensional Lebesgue measure of the set $A$ and $c_M\in(0,\infty)$ is a constant independent of $f$ and $\lambda$ (see, for instance, \cite{deGuzman1973} and \cite[Chapter 17]{Zygmund2002}).
In order to get this kind of weak type inequality for the maximal operator $P^*$,
we prove the following pointwise estimate for $P_{\bf\Delta}$ by the strong maximal function:
\begin{prop}\label{thm:pointwiseprojection}
There exists a constant $c\in(0,\infty)$ that only depends on the dimension $d$ and the spline orders ${\bf k}$ such that, for all $f\in L_1[0,1]^d$, $x\in [0,1]^d$ and all partitions $\bf\Delta$,
\[
|P_{\bf\Delta} f(x)|\leq c\cdot {\rm M_S}f(x).
\]
\end{prop}
We will now present the proof Theorem \ref{thm:ae} and defer the proofs of Propositions \ref{prop:uniformconvergence} and
\ref{thm:pointwiseprojection}.

\begin{proof}[Proof of Theorem \ref{thm:ae}]
Let $f\in L(\logplus L)^{d-1}$ and define
\[
R(f,x):=\limsup_{|{\bf\Delta}|\to 0} P_{\bf\Delta} f(x) - \liminf_{|{\bf\Delta}|\to 0} P_{\bf\Delta}f(x).
\]
Let $g\in C\big([0,1]^d\big)$. Since, by Proposition \ref{prop:uniformconvergence}, $R(g,x)\equiv 0$ for continuous functions $g$, and because $P_{\bf\Delta}$ is a linear operator,
\[
R(f,x)\leq R(f-g,x)+R(g,x)=R(f-g,x).
\]
Let $\delta>0$. Then, by Proposition \ref{thm:pointwiseprojection}, we have
\begin{align*}
|\{x: R(f,x)>\delta\}| &\leq |\{x: R(f-g,x)>\delta\}| \\
& \leq |\{x: 2c\cdot {\rm M_S}(f-g)(x)>\delta\}|.
\end{align*}
Now we employ the weak type inequality \eqref{eq:weaktype} for ${\rm M_S}$ to find
\begin{align*}
|\{x: R(f,x)>\delta\}| &\leq \\
c_M \int_{[0,1]^d}&\frac{2c\cdot|(f-g)(x)|}{\delta}\bigg(1+\logplus\frac{2c\cdot |(f-g)(x)|}{\delta}\bigg)^{d-1}\dif x.
\end{align*}
By assumption, the expression on the right-hand side of the latter display is finite.
Choosing a suitable sequence of continuous functions $(g_n)$ (first approximate
$f$ by a bounded function and then apply Lusin's theorem), the above expression tends to zero and we obtain
\[
|\{x: R(f,x)>\delta\}|=0.
\]
Since $\delta>0$ is arbitrary, $R(f,x)=0$ for a.e. $x\in [0,1]^d$. This means that $P_{\bf\Delta}f$ converges almost everywhere as $|{\bf\Delta}|\to 0$. It remains to show that this limit equals $f$ a.e.. This is obtained by a similar argument as above replacing $R(f,x)$ by $|\lim_{|{\bf\Delta}|\to 0} P_{\bf\Delta} f(x)-f(x)|$.
\end{proof}
The rest of this section is devoted to the proofs of Propositions
\ref{prop:uniformconvergence} and \ref{thm:pointwiseprojection}.

\begin{proof}[Proof of Proposition \ref{prop:uniformconvergence}]
By Corollary \ref{cor:uniformboundedness}, $P_{\bf\Delta}$ is a bounded projection operator and so, for all functions $h$ in the range of $P_{\bf\Delta}$, we have
	\begin{equation*}
		\|P_{\bf\Delta}g -g \|_\infty\leq \|P_{\bf\Delta}(g-h) \|_\infty + \|h-g\|_\infty\leq
		(1+c_{d,\bf k}) \|g-h\|_\infty.
	\end{equation*}
Taking the infimum over all such $h$,
\begin{equation}\label{eq:uniformconv} 
\|P_{\bf\Delta}g-g\|_\infty\leq (1+c_{d,{\bf k}})\cdot E_{\bf\Delta}(g),
\end{equation}
where $E_{\bf\Delta}(g)$ is the error of best approximation of $g$ by splines in
the span of tensor product B-splines $(N_{\bf i})_{{\bf 1}\leq{\bf i}\leq {\bf
n}}$. It is known that
\[
E_{\bf\Delta}(g)\leq c\cdot\sum_{\mu=1}^d \sup_{h_\mu\leq |\Delta_{\mu}|}\sup_{x} |(D_{h_\mu}^{k_\mu}g_{\mu,x})(x_\mu)|,
\]
where $g_{\mu,x}(s):=g(x_1,\dots,x_{\mu-1},s,x_{\mu+1},\dots,x_d)$ and $D_{h_\mu}$ is the forward difference operator with step size $h_\mu$ (see, for instance, \cite[Theorem 12.8 and Example 13.27]{Schumaker2007}). This is the sum of moduli of smoothness in each direction $\mu$ of the function $g$ with respect to the mesh diameters $|\Delta_1|,\dots,|\Delta_d|$, respectively. As these diameters tend to zero, the right-hand side of the above display also tends to zero since $g$ is continuous. Together with \eqref{eq:uniformconv} this proves the assertion of the proposition.
\end{proof}

Next we present the proof of Proposition \ref{thm:pointwiseprojection}. It is essentially a consequence of a pointwise estimate involving the Dirichlet kernel of the projection operator
	$P_{\bf\Delta}$. With the notation
	\begin{equation*}
		I_i^{(\mu)} :=  \Big[t_i^{(\mu)}, t_{i+1}^{(\mu)}\Big],\qquad
		I_{ij}^{(\mu)} := \conv\Big(I_i^{(\mu)}, I_j^{(\mu)}\Big), \qquad
		\mu\in\{1,\ldots,d\},
	\end{equation*}
	its 1-dimensional version,  where we suppress the superindex $(\mu)$,
reads as follows:
	\begin{lem}[\cite{PassenbrunnerShadrin2014}, Lemma 2.1]\label{lem:dirichlet}
		Let $K_{\Delta}$ be the Dirichlet kernel of the projection
		operator $P_{\Delta}$, i.e., $K_{\Delta}$ is defined by
		the equation
		\begin{equation*}
			P_{\Delta}f(x) = \int_0^1 K_{\Delta}(x,y)f(y)\dif
			y,\qquad f\in L_1[0,1],\, x\in [0,1].
		\end{equation*}
		Then $K_{\Delta}$ satisfies the inequality
		\begin{equation*}
			|K_{\Delta}(x,y)| \leq C \gamma^{|i-j|}
			|I_{ij}|^{-1},\qquad x\in I_i,\, y\in
			I_j,
		\end{equation*}
		where $C\in(0,\infty)$ and $\gamma\in (0,1)$ are constants that depend only
		on the spline order $k$.
	\end{lem}

\begin{proof}[Proof of Proposition \ref{thm:pointwiseprojection}]
	We first note that the estimate given in Lemma \ref{lem:dirichlet}
	carries over to the Dirichlet kernel $K_{\bf\Delta}$ of $P_{\bf\Delta}$
	for dimension $d$, which is defined by the relation  
\begin{equation}\label{eq:projdirichlet}
P_{\bf \Delta}f(x)=\int_{[0,1]^d} K_{\bf \Delta}(x,y)f(y)\dif y,\qquad f\in L_1[0,1]^d,\, x\in [0,1]^d.
\end{equation}
Indeed, since $P_{\bf\Delta}$ is the tensor product of the
1-dimensional projections $P_{\Delta_1},\ldots,P_{\Delta_d}$, the
Dirichlet kernel $K_{\bf\Delta}$ is the product of the 1-dimensional Dirichlet
kernels $K_{\Delta_1},\ldots, K_{\Delta_d}$. Thus, Lemma \ref{lem:dirichlet} implies the inequality
\begin{equation}\label{eq:ddimbound}
|K_{\bf\Delta}(x,y)|\leq C \gamma^{|{\bf i}-{\bf j}|_1}|I_{\bf ij}|^{-1},\qquad x\in I_{\bf i},\,y\in I_{\bf j},
\end{equation}
where we set 
\begin{equation*}
	|{\bf i}-{\bf j}|_1 := \sum_{\mu=1}^d |i_\mu - j_\mu|,\qquad
	I_{\bf i} := \prod_{\mu=1}^d I_i^{(\mu)},\qquad I_{\bf ij} :=
	\prod_{\mu=1}^d
	I_{ij}^{(\mu)},
\end{equation*}
and $C\in(0,\infty)$ and $\gamma\in (0,1)$ are constants only depending on $d$ and
$\bf k$.


Let $x\in [0,1]^d$ and $\bf i$ be such that $x\in I_{\bf i}$ and $|I_{\bf i}|>0$. 
By Equation \eqref{eq:projdirichlet},
\begin{align*}
|P_{\bf\Delta}f(x)|= \bigg|\int_{[0,1]^d}K_{\bf\Delta}(x,y)f(y)\dif y\bigg| =
\bigg| \sum_{\bf 1 \leq \bf j\leq \bf n} \int_{I_{\bf j}}
K_{\bf\Delta}(x,y)f(y)\dif y\bigg|.
\end{align*}
Using estimate \eqref{eq:ddimbound} on the Dirichlet kernel, we obtain 
\begin{align*}
|P_{\bf\Delta}f(x)|\leq C\sum_{{\bf 1}\leq{\bf j}\leq{\bf n}} \frac{\gamma^{|{\bf i}-{\bf j}|_1}}{|I_{\bf ij}|} \int_{I_{\bf j}} |f(y)|\dif y,
\end{align*}
where $C\in(0,\infty)$ is the constant in \eqref{eq:ddimbound}. 
Since $I_{\bf j}\subset I_{\bf ij}$ and $x\in I_{\bf i}\subset I_{\bf ij}$, we conclude
\[
|P_{\bf\Delta}f(x)|\leq C \sum_{{\bf 1}\leq{\bf j}\leq{\bf n}} \gamma^{|{\bf i}-{\bf j}|_1} {\rm M_S}f(x),
\]
which, after summing a geometric series, concludes the proof. 
\end{proof}

\section{Optimality of the result}\label{sec:negative}
In this section, we prove the optimality result, Theorem
\ref{thm:saks-extension}. The choice of the function $\varphi$ is based
on the following result of Saks \cite{Saks1935}:
\begin{thm}\label{thm:saks_version1}
For any function $\sigma: [0,\infty)\to[0,\infty)$ with $\liminf_{t\to\infty}\sigma(t)=0$ there exists a non-negative function $\varphi:=\varphi_\sigma$ on $[0,1]^d$ such that
\begin{enumerate}
\item the function $\sigma(\varphi)\cdot\varphi\cdot(\logplus \varphi)^{d-1}$ is integrable,
\item for all $x\in [0,1]^d$,
	\begin{equation*}
		\limsup_{\diam I\to 0,\, I\ni
		x}\frac{1}{|I|}\int_{I}\varphi(y)\dif y = \infty,
	\end{equation*}
	where $\limsup$ is taken over all $d$-dimensional rectangles $I$, which
	are parallel to the coordinate axes and contain the point $x$.
\end{enumerate}
\end{thm}
We will show that the same function $\varphi$, constructed in the proof of the previous theorem, also has the properties stated in Theorem \ref{thm:saks-extension}. 
The definition of $\varphi$ rests on a construction due to H.~Bohr that appears
in the first edition of \cite[pp.
689-691]{Caratheodory1927} from $1918$ for dimension $d=2$. 
Let us begin by recalling Bohr's construction and Saks' definition of the function $\varphi$.

\subsection*{Bohr's construction}
Let $N\in\mathbb N$ and $S:=[a_1,b_1]\times [a_2,b_2]\subset\mathbb R^2$ be a
rectangle. Using the splitting parameter $N$, we define subsets of this rectangle as follows: 
\[
I_j^{(1)}:= \Big[a_1,a_1+\frac{j(b_1-a_1)}{N}\Big]\times \Big[a_2,a_2+\frac{b_2-a_2}{j}\Big],\qquad 1\leq j\leq N.
\]
The part $S\setminus\bigcup_{j=1}^N I_j^{(1)}$ consists of $N-1$ disjoint rectangles to which we apply the same splitting as we did with $S$ (see Figure \ref{fig:In}). This procedure is carried out until the area of the remainder is less than $|S|/N^2$. The remainder is again a disjoint union of rectangles $J^{(1)},\dots,J^{(r)}$. Thus, we obtain a sequence of rectangles whose union is $S$,
\begin{equation}\label{eq:bohr1}
I_1^{(1)},\ldots, I_N^{(1)}\ ;\ I_1^{(2)},\ldots, I_N^{(2)}\ ;\ \cdots\ ;\ I_1^{(s)},\ldots, I_N^{(s)}\ ;\ J^{(1)},\ldots, J^{(r)}.
\end{equation}

\begin{figure}[ht]
\begin{center}
\begin{tikzpicture}[scale=1.5]
\draw (0,0) rectangle (5,5) ;
\draw (0,0) rectangle (1,5) node[anchor=north east]{$I_1^{(1)}$};
\draw (0,0) rectangle (2,2.5) node[anchor=north east]{$I_2^{(1)}$};
\draw (0,0) rectangle (3,1.67) node[anchor=north east]{$I_3^{(1)}$};
\draw (0,0) rectangle (4,1.25) node[anchor=north east]{$I_4^{(1)}$};
\draw (0,0) rectangle (5,1) node[anchor=north east]{$I_5^{(1)}$};
\draw[pattern=north west lines] (0,0) rectangle (1,1);
\node at (0,0.5)[anchor=east]{$\delta^{(1)}$};
\draw[dashed] (2,2.5)--(2,5);
\draw[dashed] (3,1.67)--(3,5);
\draw[dashed] (4,1.25)--(4,5);
\draw (1,2.5) rectangle (1.2,5) node[anchor=north west]{$I_1^{(2)}$};
\draw (1,2.5) rectangle (1.4,3.75) node[inner sep=1pt, anchor=south west]{$I_2^{(2)}$};
\draw (1,2.5) rectangle (1.6,3.33) node[inner sep=1pt, anchor=south west]{$I_3^{(2)}$};
\draw (1,2.5) rectangle (1.8,3.125) node[inner sep=1pt, anchor=south west]{$I_4^{(2)}$};
\draw (1,2.5) rectangle (2,3) node[anchor=west]{$I_5^{(2)}$};
\draw[pattern=north west lines] (1,2.5) rectangle (1.2,3);
\node at (1,2.75)[anchor=east]{$\delta^{(2)}$};
\end{tikzpicture}
\end{center}
\caption{First sets in the enumeration \eqref{eq:bohr1} for $N=5$.}
\label{fig:In}
\end{figure}
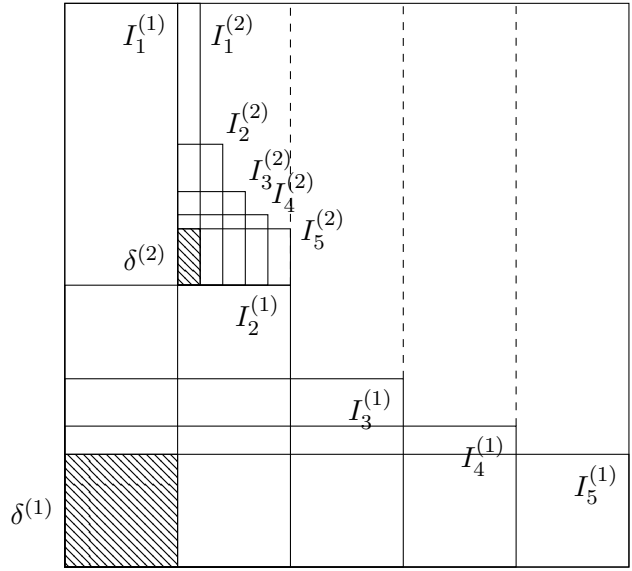

We can generalize this construction to arbitrary dimensions $d$ as follows:
first, notice that the corners of the rectangles $I_j^{(1)}$, $1\leq j\leq N$, lie on the curve
$(x-a_1)(y-a_2)=(b_1-a_1)(b_2-a_2)/N=|S|/N$. Given a rectangle
$S:=[a_1,b_1]\times\cdots \times [a_d,b_d]$, $d>2$, we consider rectangles
similar to $I_j^{(1)}$ whose
corners lie on the variety $(x_1-a_1)(x_2-a_2)\cdots (x_d-a_d)=|S|/N^{d-1}$. For
$a_1=\cdots =a_d=0$ and $b_1=\cdots = b_d=1$, we
can write those rectangles using $d-1$ parameters as
\begin{equation*}
	I_{j_1,\ldots, j_{d-1}} := \Big[0,\frac{j_1}{N}\Big] \times \cdots \times \Big[0,\frac{j_{d-1}}{N}\Big]
	\times \Big[0,\frac{1}{j_1\cdots j_{d-1}}\Big]
\end{equation*}
for $1\leq j_1,\ldots,j_{d-1}\leq N.$
The volume of the union over all those sets is approximately 
\begin{equation*}
\bigg(	\frac{\log N}{N}\bigg)^{d-1},
\end{equation*}
which can be seen by integration of the function $x_d = \big( N^{d-1} x_1\cdots
x_{d-1}\big)^{-1}$ over the rectangle $[1/N,1]^{d-1}$.
In what follows, it is important that in Bohr's construction we only 
choose those rectangles  $I_{j_1,\ldots,j_{d-1}}$ 
for which the product $j_1\cdots j_{d-1}$ is less than or equal to $N$
so that the volume $V_1$ of their union is still
approximately $N^{1-d}\log ^{d-1}N$ while the volume $V_2$ of their intersection equals
$N^{-d}$. Therefore, the quotient $V_1/V_2$ is of the order $N\log^{d-1} N$. This
is crucial for the construction of the function $\varphi$ in Theorem
\ref{thm:saks_version1}.

The function $\varphi$ from Theorem \ref{thm:saks_version1} is
	constructed in 
	\cite{Saks1935} in such a way that 
it satisfies the following additional properties:
\begin{thm}\label{thm:saks_version2}
	The function $\varphi$ from Theorem \ref{thm:saks_version1} can be chosen in such a way that
	there exist
	a sequence $(\varepsilon_i)_{i\in\mathbb N}\in(0,\infty)^{\mathbb N}$ and a sequence $(\mathcal{C}_i)_{i\in\mathbb N}$ of rectangular coverings of $[0,1]^d$ such that
\begin{enumerate}
\item the function $\sigma(\varphi)\cdot\varphi\cdot(\logplus\varphi)^{d-1}$ is integrable,
\item the sequence $(\varepsilon_i)_{i\in\mathbb N}$ converges to $0$,
\item for each $i\in\mathbb N$, $\mathcal{C}_i=(R_{ij})_{j=1}^{M_i}$ with
	$\bigcup_{j=1}^{M_i}R_{ij}=[0,1]^d$ we have $\diam R_{ij}< 1/i$ and 
\[
\frac{1}{|R_{ij}|} \int_{R_{ij}} \varphi(x) \dif x > \varepsilon_i^{-1},\qquad
\text{for all } j\in \{1,\ldots, M_i\},
\]
\item for each $i\in \mathbb N$ there exist $L_i,N_i\in\mathbb
	N$ and 
	a partition $(S_{ij})_{j = 1}^{L_i}$ of the unit cube
	$[0,1]^d$ consisting of rectangles with diameter $\leq 1/i$ such that for all
	$j\in\{1,\ldots, L_i\}$, the
	subcollection of rectangles in $\mathcal C_i$
	that intersect $S_{ij}$ is given by the rectangles in \eqref{eq:bohr1} (or its higher
	dimensional analogue)
	corresponding to $S_{ij}$ and the splitting parameter $N_i$.
\end{enumerate}
\end{thm}

Let $P_I$ be the orthogonal projection operator onto the space of $d$-variate
polynomials of order $(k_1,k_2,\ldots,k_d)$ on the
rectangle $I$. We now use Remez' inequality to prove that $|P_I\varphi|$ is large on a large subset of $I$ as long as
$\frac{1}{|I|}\int_I \varphi\dif y$ is large enough. This is the first important step in proving (ii) of Theorem \ref{thm:saks-extension}. 
\begin{lem}\label{prop:projpointwise}
Let $I\subset \mathbb R^d$ be a rectangle. Then, there exists a constant $c_{\bf
k}\in(0,\infty)$ only
depending on the polynomial orders ${\bf k}=(k_1,\ldots k_d)$ so that for all
positive functions $f$ on $I$ there exists a subset $A\subset I$
with measure $|A| \geq |I|/2$ such that, for all $x\in A$,
\begin{equation*}
	|P_I f(x)| \geq  \frac{c_{\bf k}}{|I|} \int_I f(y)\dif y.
\end{equation*}
\end{lem}
\begin{proof}
$P_I$ is the orthogonal projection operator onto the space of $d$-variate
polynomials of order $(k_1,k_2,\ldots k_d)$ on $I$. Therefore, the characteristic function $\chi_I$ is 
contained in the range of $P_I$ and  we have
\begin{equation*}
\langle P_If, \chi_I\rangle=\langle f,\chi_I\rangle.
\end{equation*}
Hence, in fact, $\| P_I f\|_{L_\infty(I)} \geq |I|^{-1}\int_I
f(y)\dif y$.
Consequently, Corollary \ref{cor:remez} implies the assertion.
\end{proof}

Considering the properties of $\varphi$ in Theorem \ref{thm:saks_version2}, the
previous proposition applied to $\varphi$ shows that for any element $I\in \mathcal C_i$, there exists a
subset $A:=A(I)\subset I$ with measure $\geq |I|/2$, on which $|P_I \varphi|\geq
c/\varepsilon_i$ for a constant $c\in(0,\infty)$ only depending on the polynomial orders
$(k_1,\ldots,k_d)$. In  Lemma 
\ref{lem:Ajbig}, we ensure that the union over those sets
$A$ still has large enough measure relatively to the measure of the union over all $I\in
\mathcal C_i$. In order to obtain that, we will use the special structure indicated by
	Bohr's construction as well as (iv) of Theorem \ref{thm:saks_version2}.

\begin{lem}\label{lem:Ajbig}
	For all $j_1,\ldots, j_{d-1}\in \{1,\ldots, N\}$, let
	\begin{equation*}
		I_{j_1,\ldots, j_{d-1}} =\Big[0,\frac{j_1}{N}\Big] \times \cdots \times \Big[0,\frac{j_{d-1}}{N}\Big]
	\times \Big[0,\frac{1}{j_1\cdots j_{d-1}}\Big]
	\end{equation*}
	and $\Lambda = \{(j_1,\ldots, j_{d-1}) : j_1\cdots j_d\leq N\}$.
	For $\lambda\in \Lambda,$ let $A_\lambda\subset I_\lambda$ be a Borel
	measurable subset of $I_\lambda$ so that 
	\begin{equation*}
		|A_\lambda|\geq c|I_\lambda| = \frac{c}{N^{d-1}}
	\end{equation*}
	for some absolute constant $c\in(0,\infty)$. Then there exist constants
	$c_1,c_2\in(0,\infty)$
	that depend only on $c$ and $d$ so that 
	\begin{equation*}
		\Big| \bigcup_{\lambda\in\Lambda} A_\lambda\Big| \geq c_2
		\Big(\frac{\log N}{N}\Big)^{d-1}\geq c_1 \Big|
		\bigcup_{\lambda\in\Lambda} I_\lambda\Big|.
	\end{equation*}
\end{lem}
\begin{proof}
	Let $M\in\mathbb N$ to be specified later and define $q:=1/M$. Define the index set 
	\begin{equation*}
		\Gamma = \{ (M^{k_1}, \ldots, M^{k_{d-1}}) \in \Lambda :
		k_1,\ldots, k_{d-1}\in \mathbb N_0\}.
	\end{equation*}
	Then, we can estimate
	\begin{align*}
		\Big| \bigcup_{\lambda\in\Lambda} A_\lambda\Big| \geq 
		\Big| \bigcup_{\lambda\in\Gamma} A_\lambda\Big| &\geq
		\sum_{\lambda\in\Gamma} |A_\lambda| - \frac{1}{2}
		\sum_{\substack{\lambda,\mu\in\Gamma\\ \lambda\neq \mu}}
		|A_{\lambda} \cap A_\mu| \\
		&\geq c \sum_{\lambda\in\Gamma} |I_\lambda| - \frac{1}{2}
		\sum_{\substack{\lambda,\mu\in\Gamma\\ \lambda\neq \mu}}
		|I_{\lambda} \cap I_\mu|.
	\end{align*}
	Now we observe that $\card[\Gamma]=\card[\{k\in \mathbb N_0^s : \sum_{j=1}^s
	k_j\leq L\}]={\lfloor L\rfloor+s\choose s}$ where $L=\log_M N$, $s=d-1$
	and $\lfloor L\rfloor$ denotes the largest integer smaller than or
	equal to $L$.
	Therefore, $\card[\Gamma]\geq C_s \log_M^{d-1} N$ for some positive constant
	$C_s$ depending only on $s$.
	Thus,
	\begin{equation}\label{eq:lowermeasurestimate}
		\Big| \bigcup_{\lambda\in\Lambda} A_\lambda\Big| \geq c \cdot C_{d-1} 
		\Big(\frac{\log_M N}{N}\Big)^{d-1} -
		\frac{1}{2}\sum_{\substack{\lambda,\mu\in\Gamma\\ \lambda\neq \mu}}
		|I_{\lambda} \cap I_\mu|.
	\end{equation}
	Next, observe that if $\lambda,\mu\in\Gamma$ have the form
	$\lambda = (M^{\ell_1},\ldots,M^{\ell_{d-1}})$ and
	$\mu = ( M^{m_1},\ldots, M^{m_{d-1}} )$,
	\begin{equation*}
		|I_\lambda \cap I_\mu| = N^{1-d}q^{\sum_{i=1}^{d-1} \big(
	\max(\ell_i,m_i) - \min(\ell_i,m_i)\big)},
	\end{equation*}
	which, by summing geometric series and noting that the condition $\lambda\neq
	\mu$ implies the existence of  at least one index $i\in\{1,\ldots,d-1\}$
	so that $\lambda_i\neq \mu_i$, yields
	\begin{equation*}
		\sum_{\substack{\lambda,\mu\in\Gamma \\ \lambda\neq \mu}}
		|I_\lambda\cap I_\mu| \leq
		\frac{q}{(1-q)^{d-1}}\sum_{\lambda\in\Gamma} N^{1-d} \leq
		\frac{q}{(1-q)^{d-1}} \Big(\frac{\log_M N}{N}\Big)^{d-1}.
	\end{equation*}
	Inserting this inequality in \eqref{eq:lowermeasurestimate}, we obtain
	\begin{equation*}
		\Big| \bigcup_{\lambda\in\Lambda}A_\lambda\Big| \geq \Big(
		c\cdot C_{d-1}
		- \frac{q}{2(1-q)^{d-1}}\Big)\cdot \Big(\frac{\log_M
		N}{N}\Big)^{d-1}.
	\end{equation*}
	We can choose $M=1/q$ (depending only on $c$ and $d$) sufficiently large
	to guarantee that $c \cdot C_{d-1}
	- \frac{q}{2(1-q)^{d-1}}\geq c\cdot C_{d-1}/2$. Then the assertion of
	the lemma follows
	with the choice $c_2 = c\cdot C_{d-1}/(2 \log^{d-1} M)$.
\end{proof}

Bringing together the above facts, we are now able to prove our optimality result:
\begin{proof}[Proof of Theorem \ref{thm:saks-extension}] We subdivide the proof
	into two parts. In the first part, we show that for all
		points $x$ in a set of positive measure
	there exists a sequence
	$(I_n)$ of intervals containing $x$ whose
	measure tends to zero and such that $|P_{I_n}\varphi(x)|\to\infty$.
Based on that observation, we construct the desired sequence of partitions in
the second step.
\vskip 1mm
\noindent\textbf{Step 1:}
Since Theorem \ref{thm:saks_version1} proves the integrability condition (i) of Theorem \ref{thm:saks-extension}, we only
need to prove (ii), i.e., the existence of a set $B\subset [0,1]^d$ with positive Lebesgue measure and of a
sequence $({\bf\Delta}_n)$ of partitions such that for all $x\in B$, $\limsup_{n\to\infty}
|P_{{\bf\Delta}_n}\varphi(x)|=\infty$.
We fix $i\in\mathbb N$ and consider the corresponding covering 
$\mathcal C_i$ of $[0,1]^d$ from Theorem \ref{thm:saks_version2}. 
Then, we define 
\begin{align*}
	B_i:=\big\{x\in[0,1]^d\,:\, &\text{there exists a rectangle $I\in\mathcal C_i$ with $x\in I$}\\
	&\text{and $|P_I\varphi(x)|\geq c_{\bf k}/\varepsilon_i$}\big\},
\end{align*}
where $c_{\bf k}\in(0,\infty)$ is the constant that appears in Lemma \ref{prop:projpointwise} and $(\varepsilon_i)$ is the
sequence from Theorem \ref{thm:saks_version2}. Recall that $\varepsilon_i\to 0$ as $i\to\infty$. We will show that
$|B_i|\geq c>0$ for all $i\in\mathbb N$ and some suitable constant $c\in(0,\infty)$.

Let $I\in\mathcal C_i$.  Due to Theorem \ref{thm:saks_version2}, we have $\diam
I\leq 1/i$ and
\[
\frac{1}{|I|}\int_I \varphi \dif x\geq \varepsilon_i^{-1}.
\]
Thus, Lemma \ref{prop:projpointwise} provides a set $A(I)\subset I$ with
$|A(I)|\geq |I|/2$ such that, for all $x\in
A(I)$,
\[
	|P_I\varphi(x)|\geq \frac{c_{\bf k}}{\varepsilon_i}\,. 
\]
This means that $A(I) \subset B_i$.
For fixed $j$ let $(I_{m}^{(\ell)}), (J^{(\ell)})$ be the collections of 
rectangles \eqref{eq:bohr1} contained in $\mathcal C_i$ 
forming a covering of $S_{ij}$ (see Theorem \ref{thm:saks_version2}, part (iv)).
As a consequence of the latter bound,  Lemma \ref{lem:Ajbig} and the fact
that the rectangles $J^{(\ell)}$ are disjoint, we find
\begin{align*}
	|S_{ij}\cap B_i|&\geq
	\sum_{\ell}\Big|\bigcup_{m=1}^{N_i}A(I_{m}^{(\ell)})\Big|+\sum_{\ell}
	|A(J^{(\ell)})| \\
	&\geq c_1
	\sum_{\ell}\Big|\bigcup_{m=1}^{N_i}I_{m}^{(\ell)}\Big|+\frac{1}{2}\sum_{\ell}
	\big|J^{(\ell)}\big|\geq c_2\,|S_{ij}|\,, 
\end{align*}
where $c_2:=\min\big\{c_1,\frac{1}{2}\big\}$. Consequently, 
\begin{align*}
	|B_i|=\sum_{j=1}^{L_i} |S_{ij}\cap B_i|\geq c_2\sum_{j=1}^{L_i} |S_{ij}|=c_2\, \big|[0,1]^d\big|=c_2\,.
\end{align*}
Since all sets $B_i$ satisfy this uniform lower bound, the set $B:= \limsup_n B_n$ has positive measure as well, because
\[
|B|=\lim_n\bigg|\bigcup_{m\geq n} B_m\bigg|\geq \limsup_n |B_n|\geq c>0\,.
\]
\vskip 1mm
\noindent\textbf{Step 2:} We now proceed with the construction of the desired sequence of partitions $({\bf\Delta}_n)$. 
	Let $(R_{ij})_{j=1}^{M_i}$ be the rectangles contained
	in the collection $\mathcal C_i$.
	For $1\leq j\leq M_i$, we define 
	the partition ${\bf\Delta}^{(i,j)}= (\Delta_1^{(i,j)},\ldots, \Delta_d^{(i,j)})$
	such that each $R_{ij}$ is a grid
	point interval of ${\bf\Delta}^{(i,j)}$ and, for $\mu\in\{1,\ldots, d\}$,
	the $\mu$-th coordinate projection of the vertices of
	$R_{ij}$ has
	multiplicity $k_\mu$ in the partition $\Delta_\mu^{(i,j)}$. We give this multiplicity condition in order
	to have, for all $x\in R_{ij}$, 
	\begin{equation*}
		P_{{\bf \Delta}^{(i,j)}} f(x) = P_{R_{ij}} f(x), \qquad f\in L_1[0,1]^d.
	\end{equation*}
	Other knots of the partition ${\bf\Delta}^{(i,j)}$ 
	are chosen arbitrarily, with the only condition 
	$|{\bf\Delta}^{(i,j)}|\leq 1/i$. Observe that this is possible since
	$\diam R_{ij}\leq
	1/i$. Now we define the sequence $({\bf\Delta}_n)$ as 
	\begin{equation*}
		({\bf\Delta}_n) := \big({\bf\Delta}^{(1,1)},\ldots,
		{\bf\Delta}^{(1,M_1)},{\bf\Delta}^{(2,1)},\dots,{\bf\Delta}^{(2,M_2)},\dots\big).
	\end{equation*}
	Observe that this sequence of partitions is not nested. 
	In order to prove the assertion of the theorem, we fix some $x\in B$. By definition of $B$,
	for infinitely many indices
	$i\in\mathbb N$, there exists a rectangle $R_{i\ell_i}$ in the 
	collection $\mathcal C_i$ such that $x\in R_{i\ell_i}$,
	$\diam R_{i\ell_i} \leq 1/i$\,, and
	$|P_{{\bf\Delta}^{(i,\ell_i)}}\varphi(x)|=
	|P_{R_{i\ell_i}}\varphi(x)|\geq c_{\bf k}/\varepsilon_i$. Therefore, since
	$\varepsilon_i\to 0$, we have, for all $x\in B$,
	\begin{equation*}
		\limsup_{n\to\infty} |P_{{\bf\Delta}_n}\varphi(x)| = \infty.
	\end{equation*}
This completes the proof of the theorem.
\end{proof}

\section{Final remarks and open problems}\label{sec:final remarks}
It is natural to ask whether the rather general structure of the partitions $\bf\Delta$, whose mesh
diameter tends to zero in Theorem \ref{thm:ae}, can be relaxed to obtain a.e.
convergence for a larger class than  $L(\logplus L)^{d-1}$. A result in this direction is supported by the fact that in the case of piecewise constant functions, we get a.e. convergence for all $L_1$-functions provided the underlying sequence of partitions is nested.
This holds as the sequence of projection operators applied to an
$L_1$-function then forms a martingale. 
Although at first it seems that approaching this problem for general spline
orders under the same framework should lead to a positive or negative answer, we must say that it is far from clear
if such a result holds. On the other hand, it is  unclear how to generalize Saks' construction from \cite{Saks1935}
to this setting, since the sequence of partitions constructed in the proof of Theorem \ref{thm:saks-extension} is not nested.

We close this work with the following open problem:

\begin{problem}
	Is it true that the a.e. convergence in Theorem \ref{thm:ae} holds for all $f\in L_1$ under the assumption that the sequence of partitions is nested?
\end{problem}

\subsection*{Acknowledgments}
We are grateful to the anonymous referees for their valuable suggestions that improved the quality of the paper.

M. Passenbrunner was supported by the Austrian Science Fund, FWF projects P 23987-N18 and P 27723-N25. J. Prochno was supported in parts by the Austrian Science Fund, FWFM 1628000. 

\bibliographystyle{plain}
\bibliography{tensor}
\end{document}